\documentstyle[12pt]{article}

\begin{document}

\input{psfig}

%blackboard bold for 10 pt
%\font\bbbld=msym10
%blackboard bold for 11 pt
%\font\bbbld=msym10 scaled\magstephalf
%blackboard bold for 12 pt
\font\bbbld=msbm10 scaled\magstep1
\newcommand{\bfR}{\hbox{\bbbld R}}
\newcommand{\bfC}{\hbox{\bbbld C}}
\newcommand{\bfZ}{\hbox{\bbbld Z}}
\newcommand{\bfH}{\hbox{\bbbld H}}
\newcommand{\bfQ}{\hbox{\bbbld Q}}
\newcommand{\bfN}{\hbox{\bbbld N}}
\newcommand{\bfP}{\hbox{\bbbld P}}
\newcommand{\bfT}{\hbox{\bbbld T}}
\def\Sym{\mathop{\rm Sym}}
\newcommand{\suchthat}{\mid}
\newcommand{\halo}[1]{\Int(#1)}
\def\Int{\mathop{\rm Int}}
\def\Re{\mathop{\rm Re}}
\def\Im{\mathop{\rm Im}}
\newcommand{\union}{\cup}
\newcommand{\goesto}{\rightarrow}
\newcommand{\bdy}{\partial}
\newcommand{\n}{\noindent}
\newcommand{\p}{\hspace*{\parindent}}

\newtheorem{theorem}{Theorem}[section]
\newtheorem{assertion}{Assertion}[section]
\newtheorem{proposition}{Proposition}[section]
\newtheorem{lemma}{Lemma}[section]
\newtheorem{definition}{Definition}[section]
\newtheorem{claim}{Claim}[section]
\newtheorem{corollary}{Corollary}[section]
\newtheorem{observation}{Observation}[section]
\newtheorem{conjecture}{Conjecture}[section]
\newtheorem{question}{Question}[section]

\newbox\qedbox
\setbox\qedbox=\hbox{$\Box$}
\newenvironment{proof}{\smallskip\noindent{\bf Proof.}\hskip \labelsep}%
			{\hfill\penalty10000\copy\qedbox\par\medskip}
\newenvironment{proof1}{\smallskip\noindent{\bf Proof of Theorem 2.1.}
			\hskip \labelsep}%
			{\hfill\penalty10000\copy\qedbox\par\medskip}
\newenvironment{proof2}{\smallskip\noindent{\bf Proof of Corollary 2.1.}
			\hskip \labelsep}%
			{\hfill\penalty10000\copy\qedbox\par\medskip}
\newenvironment{remark}{\smallskip\noindent{\bf Remark.}\hskip \labelsep}%
			{\hfill\penalty10000\copy\qedbox\par\medskip}
\newenvironment{example}{\smallskip\noindent{\bf Example.}\hskip \labelsep}%
			{\hfill\penalty10000\copy\qedbox\par\medskip}
\newenvironment{proofspec}[1]%
		      {\smallskip\noindent{\bf Proof of #1.}\hskip \labelsep}%
			{\nobreak\hfill\hfill\nobreak\copy\qedbox\par\medskip}
\newenvironment{acknowledgements}{\smallskip\noindent{\bf Acknowledgements.}%
	\hskip\labelsep}{}

\setlength{\baselineskip}{1.2\baselineskip}

\title{Minimal Surfaces with Planar Boundary Curves}

\author{Wayne Rossman}

\maketitle

\section{Introduction}

In 1956, Shiffman \cite{Sh} proved that any compact minimal annulus with two 
convex boundary curves (resp. circles) in parallel planes is foliated by 
convex planar curves (resp. circles) in the intermediate planes.  
In 1978, Meeks conjectured that the assumption the minimal surface 
is an annulus is unnecessary \cite{M}; that is, 
he conjectured that any compact connected 
minimal surface with two planar convex 
boundary curves in parallel planes must be an annulus.  

Partial results have been proven in the direction of this conjecture.  Schoen 
\cite{Sc1} proved the Meeks conjecture in the 
case where the two boundary curves share a 
pair of reflectional symmetries in planes perpendicular to the planes 
containing the boundary curves.  
Another interesting result related to the Meeks conjecture has been 
proven by Meeks and White (Theorem~1.2, \cite{MW2}).  
Recall that a minimal surface $M$ is called {\em stable} if, with respect 
to any normal variation that vanishes on $\partial M$, the second derivative 
of the area functional is positive.  The minimal surface is 
{\em unstable} if there exists such a variation with negative second 
derivative for the area functional, and it is {\em almost-stable} 
if the second derivative is nonnegative for all such 
variations and is zero for some nontrivial variation.  Recall also that a subset 
of $\bfR^3$ is called {\em extremal} if it is contained in the boundary 
of its convex hull.  The result of Meeks and White is that if 
$\Gamma$ is an extremal pair of smooth disjoint convex curves in 
distinct planes, then exactly one of the following holds: 
\newcounter{num}
\begin{list}%
{\arabic{num})}{\usecounter{num}\setlength{\rightmargin}{\leftmargin}}
\item $\Gamma$ is not the boundary of any connected compact minimal 
surface, with or without branch points.
\item $\Gamma$ is the boundary of exactly one minimal annulus and this 
annulus is almost-stable.  In this case, $\Gamma$ bounds no other connected 
branched minimal surface.
\item $\Gamma$ is the boundary of exactly two minimal annuli, one 
stable and one unstable.  
\end{list}

Other partial results toward the Meeks conjecture have been 
proven by Meeks and 
White for stable surfaces \cite{MW1}, \cite{MW2}.  
They have proven the conjecture 
for stable and almost-stable minimal surfaces that have two convex boundary 
curves lying in parallel planes such that 
\begin{list}%
{\arabic{num})}{\usecounter{num}\setlength{\rightmargin}{\leftmargin}}
\item the two boundary curves have a common plane of reflective 
symmetry perpendicular to the planes containing them, or 
\item the two boundary curves are reflected into each other by a plane 
parallel to the planes containing them.
\end{list}
The first of these two conditions has been extended by Meeks and White 
to boundary curves lying in nonparallel planes, but still forming an extremal 
set.  The second of these two conditions is generalized to nonparallel 
planes by Theorem~2.1 in the next section.  

In section~3 we consider a more general setting: 
compact connected minimal surfaces, with a pair of boundary curves 
(not necessarily convex) in 
distinct planes, that have least-area amongst all orientable surfaces with the same 
boundary.  When the planes containing these two boundary curves 
are either parallel or ``sufficiently close'' to parallel, and when 
the boundary curves themselves are ``sufficiently close'' 
to each other, one 
can draw specific conclusions about the geometry and topology of the surfaces, 
as in Theorem~3.1 and Theorem~3.2.  
(These theorems are formal statements about surfaces that can be physically 
realized by experimentation with wire frames and soap films.)  

In the final section~4 we state two well known results on the existence of 
compact minimal surfaces with planar boundary curves in parallel 
planes.  These results are easily proven using the 
maximum principle for minimal surfaces, 
% and were known by Schoen and Kusner (see \cite{Ku}) Schoen; 
but do not appear elsewhere in the literature.  
%are you sure they aren't elsewhere in the literature???????? 
A corollary of these results is a generalization of a result of Nitsche 
\cite{N1}: Let $M$ be 
any compact minimal annulus with two planar boundary curves of 
diameters $d_1$ and $d_2$ in parallel planes $P_1$ and $P_2$; if the 
distance between $P_1$ and $P_2$ is 
$h$, then the inequality $h \leq 
\frac{3}{2}\max\{d_1,d_2\}$ is satisfied.  The corollary we prove here 
does not assume the 
minimal surface $M$ is an annulus and has the strengthened conclusion 
$h \leq \max\{d_1,d_2\}$.  
We also include a similar result for nonminimal constant mean curvature 
surfaces.

The author thanks Miyuki Koiso, Robert Kusner, and Masaaki Umehara for helpful 
conversations.  

\section{Topological Uniqueness in the Stable Case} 

Meeks and White \cite{MW1} proved the following, 
which we generalize in Theorem~2.1:  
{\it Let $\alpha$ be a smooth convex plane curve in $\bfR^3$, let $P_0$ 
be a plane parallel to the plane containing $\alpha$, and let 
${\em Ref}:\bfR^3 
\longrightarrow \bfR^3$ be reflection in the plane $ P_0 $.  If
$ \Sigma $ is a connected stable or almost-stable compact minimal surface
 with boundary $ \alpha \cup {\em Ref}(\alpha) $, then $ \Sigma $ is an 
embedded annulus.}

	Theorem 2.1 is similar to the above 
result, but we assume $\alpha$ does not lie in a plane 
parallel to $P_0$.  We assume only that $\alpha$ lies on one side of 
$P_0$, thus $\alpha$ and $\mbox{Ref}(\alpha)$ do not lie in 
parallel planes.  In order to 
state this precisely, we introduce some notation.

	Consider all closed half-planes in $ \bfR^3 $ with boundary the 
$x_1$-axis.  Let $\omega (0) $ be the half-plane containing the 
positive $x_2$-axis, and let $\omega (\theta) $ be the half-plane making an 
angle of $ \theta $ with $\omega (0) $.  (Thus the half-plane containing the 
positive $x_3$-axis is $\omega (\pi/2) $.)  Let $ R_{\theta} $ be a 
rotation of angle $\theta$ about the $x_1$-axis.  

\begin{definition}
	The wedge between $\omega(\gamma)$ and $\omega(\beta)$ 
of angle $\beta - \gamma$ in  $\bfR^3$ is $ W_{\omega,\beta} = \cup_
{(\gamma\leq\alpha\leq\beta)} \omega(\alpha) $.
\end{definition}

\begin{theorem}
	If $M\subseteq W_{\gamma,\beta} $ is a stable or 
almost-stable compact connected
minimal surface with boundary $ \partial M = C\cup R_{\beta-\gamma}(C) $, 
where $ C $ is a strictly convex curve contained in $ \omega(\gamma) $,
 then $ M $ is an embedded annulus.
\end{theorem}

	We now describe the natural free boundary problem for a wedge.
Suppose $ C $ is a strictly 
convex Jordan curve in $ \omega(\beta) $ and $ M $ is a 
compact branched minimal surface such that the boundary 
$ \partial M $ consists of 
$ C $ and a nonempty collection of immersed curves in the $x_1x_2$-plane.  
We may assume that $ 0 \leq \beta \leq \pi/2 $, for if not, by 
reflections through the $x_1x_3$-plane and $x_1x_2$-plane, we have a  
congruent problem where this is so.  
If $ M $ is orthogonal to the $x_1x_2$-plane along 
$\partial M \cap \omega(0)$, then $ M $ is called a solution of the free 
boundary value problem for $ C $ and $ W_{0,\beta} $.  
By the maximum principle (see, for example, \cite{Sc1}), 
if $M$ is such a solution, then the 
portion of $\partial M$ lying in the 
$x_1x_2$-plane actually lies in $ \omega(0) $.  
If, with respect to any normal variation 
of $ M $ that vanishes on $ C $, the second derivative of the area 
functional is positive, then $ M $ is a {\em stable} solution to the 
free boundary value problem.  Similarly, we can define when $ M $ is 
{\em unstable} or {\em almost-stable} \cite{MW1}, \cite{MW2}.  

\begin{corollary}
Suppose $ M $ is a stable or almost-stable solution to the free
boundary value problem for $ C $ and $ W_{0,\beta} $, then $ M $ is 
an embedded annulus.
\end{corollary}

%Note that $C \cup R_{-\beta}(C)$ is an extremal subset of $\bfR^3$.  

	The proof of Theorem~2.1 requires the use of the Jacobi operator and 
Jacobi fields.  Following the notation of Choe \cite{C}, three types 
of Killing vector fields in $ \bfR^3 $ are $ \phi_n $, $ \phi_l $, and 
$ \phi_p $, where these three vector fields are the variation vector fields 
produced by translating in the direction of the 
unit vector $ n $, rotating around a straight 
line $ l $, and homothetically expanding from a point 
$ p $, respectively.  Meeks and White \cite{MW1} were interested in 
$ \phi_n $, and here we are interested in $ \phi_p $.  

	Let $ S $ be a smooth immersed surface in $ \bfR^3 $.  
The {\em horizon} 
of $ S $ with respect to $ \phi_p $, denoted by $ H(S;\phi_p) $,
 is the set of all points of $ S $ at which $ \phi_p $ is a tangent vector 
of $ S $.  A connected subset $ D $ of $ S $ is called a {\em visible} set
with respect to $ \phi_p $ if $ D $ is disjoint from $ H(S;\phi_p) $.  
If $ M $ is a minimal surface in $ \bfR^3 $, then 
$ \phi_n(M^\perp) $, $ \phi_l(M^\perp) $, and $ \phi_p(M^\perp) $ (the 
projection of $ \phi_n $, $ \phi_l $, and $ \phi_p $ onto the normal 
bundle of $ M $, respectively) are Jacobi fields \cite{C}.  If $ M $ has 
a strictly proper open connected subset $ D $ such that 
$ \partial D \subseteq H(M;\phi_p) $, then it follows from the 
Smale index theorem (see, for example, \cite{C}, p199) that $ D $ is
almost-stable or unstable, and hence $ M $ is unstable.  We use this 
fact in the following proof.  

\begin{proof1}
	By a rotation of $\bfR^3$ if necessary, we may assume $ M 
\subseteq W_{-\beta,\beta}$, with 
$\partial M = C \cup R_{-2\beta}(C), \; C \subseteq \omega(\beta), \; 
R_{-2\beta}(C) 
\subseteq \omega(-\beta)$, and $ 0 < \beta \leq \pi/2$.  In fact $\beta 
< \pi/2$, since otherwise (by the maximum principle) 
$M$ would be a pair of disks, which is not 
connected.  The orthogonal projection $\Lambda$ from $\bfR^3$ 
to the $x_1x_2$-plane maps $C$ and $R_{2\beta}(C)$ to a strictly convex 
smooth Jordan curve $\Lambda(C)$ in $\omega(0)$.  
By Theorem 2 of \cite{Sc1}, $M$ is embedded and 
$ M = \{(x_1,x_2,\pm u(x_1,x_2)) \mid (x_1,x_2) 
 \in
\Omega\}$, where $\Omega$ is a compact
 region in $ \omega(0) $, and 
$u(x_1,x_2)$ is a nonnegative function defined on $\Omega$.  
Furthermore, since 
$ M_0^+ = M \cap \{x_3 \geq 0\}$ and $ M_0^- = M \cap \{x_3 \leq 0\}$ 
have locally bounded slope in their interiors \cite{Sc1}, and since the 
maximum principle immediately implies that the tangent planes of $M$ along 
$C$ and $ R_{-2\beta}(C) $ can never by vertical, 
the normal vectors of 
$ M $ are horizontal only on $ M \cap \omega(0) $.  

	Assume $ M $ is not an annulus.  Then $ \partial \Omega = \Lambda(C) 
\cup C_1 \cup \cdot\cdot\cdot \cup C_k $, where $ k > 1 $, and $ C_i $ is a 
curve in the compact region of $\omega(0)$ bounded by $ \Lambda(C)$, 
for all $i$.  Note that the curves $C_i$ are also planar geodesics 
in $M$.  We now check that each 
$ C_i $ must be strictly convex.  If not, then there exists an $i$ and a
$ q \in C_i $ such that the Gaussian 
curvature of $M$ at $q$ is zero.  Thus in a small neighborhood of $q$, 
$ M \cap T_q(M) $ is a set of at least three curves 
crossing at $q$.  This implies $ M_0^+ $ cannot be a graph, 
contradicting Theorem 2 of \cite{Sc1}.

	There must be at least one point in the interior of 
$ M_0^+ $ at which the Gaussian curvature $K$ vanishes.  
This follows from an argument identical to the argument given in the third 
paragraph of the proof of Theorem 3.1 in \cite{MW1}, so we do not 
repeat it here.  Thus there exists a point $ q \in Int(\Omega) $ such 
that the curvature 
$K$ at $ (q,u(q)) $ vanishes.  Let $ p $ be a point on the $x_1$-axis 
which is also contained in $ T_{(q,u(q))}(M)$.  
By translating if necessary we may assume 
that $p$ is the origin $\vec{0}$.  Thus $(q,u(q)) \in H(M;\phi_{\vec{0}})$.  
(If $ T_{(q,u(q))}(M)$ does not intersect the $x_1$-axis, then we can replace 
$\phi_{\vec{0}}$ with $\phi_n, n = \vec{e_1}$, and the remaining arguments 
of this proof follow through.)  

Since $G$ is a conformal 
map with branch points, $ H(M;\phi_{\vec{0}}) $ consists of smooth curve segments, 
whose endpoints meet in even numbers at isolated points in the interior 
of $M$.  
In particular, at least four such curves meet at the point $ (q,u(q)) $, 
because it is a branch point of $G$.  Note also that 
$ H(M;\phi_{\vec{0}})$ intersects each of $C$, $ R_{-2\beta}(C) $, 
and each 
$C_i$ at exactly two points, as these are strictly convex planar curves.  
Let $Z = \Lambda(H(M;\phi_{\vec{0}})) $.  
Since $Z$ is 
homeomorphic to $ H(M_0^+ ;\phi_{\vec{0}}) $, it has the same structure.  Hence 
$Z$ meets each $ C_i $ and also $\Lambda(C)$ exactly twice.  

	Now form a topological space $\hat{\Omega}$ from $\Omega$ by 
identifying each $C_i$ to a point, and note that $\hat{\Omega}$ is 
topologically a disk.  The corresponding set $\hat{Z}$ (of $Z$) in 
$\hat{\Omega}$ is a graph in which each vertex (except for the two 
vertices on $\Lambda(C) = \partial \hat{\Omega}$) 
has an even number of edges.  
Furthermore, the vertex q has at least four edges.  It follows that 
$\hat{\Omega} - \hat{Z}$ contains at least one connected component 
$\hat{\cal U}$ whose closure does not intersect 
$\Lambda(C) = \partial \hat{\Omega}$.  
Let $\cal U$ be the corresponding region in $\Omega$, and let $\tilde{\cal U} 
= \{(x_1,x_2,x_3) \in M \, | \, (x_1,x_2) \in {\cal U}\}$.  Then 
$\partial \tilde{\cal U} \subseteq H(M,\phi_{\vec{0}})$, 
and therefore $\tilde{\cal U}$ 
is an almost-stable or unstable proper subset of 
$M$, so $M$ must be unstable (by the Smale index theorem).  This contradiction 
implies that $k=1$ and $M$ is an annulus.  
\end{proof1}

\begin{proof2}
	Let $M$ satisfy the conditions of Corollary~2.1, and 
let $\mbox{Ref}(M)$ be the reflection of $M$ across $\omega(0)$.  
It follows from the Schwarz reflection principle (see, for example, \cite{O1}, 
Lemma 7.3)  
that $M \cup \mbox{Ref}(M)$ is a smooth minimal surface which 
satisfies the conditions of Theorem~2.1 and therefore is an 
embedded annulus.  By the arguments above, $\partial M$ 
intersects $\omega(0)$ along a 
single curve (since $k=1$), and $M$ is a graph over $\omega(0)$, 
so $M$ itself must be an embedded annulus.  
\end{proof2}

\begin{remark}
	It is not possible to generalize the Meeks and White result or 
Theorem 2.1 
to the case where a boundary curve has a 1-to-1 perpendicular 
projection to a strictly convex planar curve, but is not itself 
a planar curve.  A counterexample can be constructed as 
follows: Let the $x_1x_2$-plane be the plane in which the strictly convex 
projection lies.  Consider the two line segments 
connecting the points (0,0,1) and ($N$,0,1), (0,$\delta$,1) 
and ($N$,$\delta$,1), 
respectively, for some large $N$.  
Connect them at the ends with two horizontal semicircles of radius 
$\frac{\delta}{2}$.  
(The semicircles are $\{ x_1 \leq 0 \} \cap \{ x_3 = 1 \} 
\cap \{ x_1^2 + (x_2 - \frac{\delta}{2})^2 = (\frac{\delta}{2})^2 \}$ and 
$\{ x_1 \geq N \} \cap \{ x_3 = 1 \} 
\cap \{ (x_1 - N)^2 + (x_2 - \frac{\delta}{2})^2 = 
(\frac{\delta}{2})^2 \}$.)  
A small perturbation gives us a smooth strictly 
convex curve in the plane $\{x_3 = 1\}$, which is symmetric with respect to 
reflection in the plane $\{x_2 = \frac{\delta}{2}\}$.  
Along each of the two almost straight 
segments deform the third coordinate of the 
curve by making $n$ smooth vertical dips almost to the $x_1x_2$-plane, 
so that the curve remains symmetric with respect to the plane $\{x_2 = 
\frac{\delta}{2}\}$.  
Calling this resulting loop $\alpha$, it is easy to imagine a stable 
embedded minimal surface of genus $n-1$ with boundary $\alpha \cup 
\mbox{Ref}(\alpha)$, 
where $\mbox{Ref}$ is reflection through the $x_1x_2$-plane (see Figure~1).  
\end{remark}

\begin{figure}
\centerline{
        \hbox{
		\psfig{figure=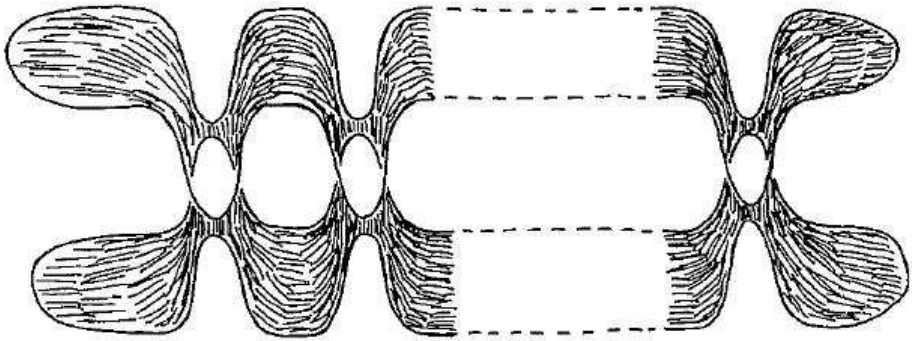,width=5in}
	}
}
	\caption{A Counterexample with Nonplanar Boundary Curves}
\end{figure}

\section{Topological Uniqueness for Least-Area Surfaces}  

In this section, we extend our considerations to include 
non-convex boundary curves.

Let $\alpha$ and $\beta$ be two $C^2$ planar Jordan curves (not 
necessarily convex) in parallel planes $P_0$ and $P_1$, respectively.  
Without loss of generality, we may assume $P_0 = \{x_3 =0\}$ and
$P_1 = \{x_3 = 1\}$.  
Let $\alpha(t)$ be a one-parameter family of planar curves 
(lying in planes $P_t = 
\{x_3=t\}$ parallel to $P_0$) that are vertical translations 
of the curve $\alpha = \alpha(0)$ in the direction of the plane 
$P_1$ at constant unit speed 
(making $\frac{\partial}{\partial t}\alpha(t)$ a constant vector) so that 
$\alpha(1)$ and $\beta$ both lie in $P_1$.  Consider the points of 
$\beta$ where $\alpha(1)$ and $\beta$ intersect.  Assume these intersections 
are transverse, and that there are a finite number $n$ of them.  We 
will call these points $\{p_1,...,p_n\}$ on 
$\beta$ the {\em crossing points} of $\alpha$ and $\beta$.  
Since $\alpha(1)$ intersects $\beta$ transversely, 
every vertex in the graph $\alpha(1) \cup \beta$ has exactly four 
edges emanating from it.  Furthermore, to each of the connected components of 
$P_1 \setminus \{\alpha(1) \cup \beta\}$ we can uniquely 
assign either a plus sign or a minus sign 
so that the following statement is true: 
No two adjacent components have the same sign, and the single 
component which is not compact has a minus sign (see Figure~2).  
Define $s$ to be the number of 
components with a plus sign.  Define $\cal W$ to be the union of 
the open components of $P_1 \setminus \{\alpha(1) \cup \beta\}$ which are assigned
a plus sign.  

For any set $B \subseteq P_1$, let $C(B) = \{(x_1,x_2,x_3) \in \bfR^3 \; | 
\; (x_1,x_2,1) \in B \}$ be the cylindrical domain in $\bfR^3$ over the 
set $B$.  Let $B_{\epsilon,P_1} (p)$ be an open 
$\epsilon$-ball in $P_1$ about a point 
$p \in P_1$.  
%If $p \in \alpha(1) \cap \beta$, we shall refer to the 
%solid cylinder $C(B_{\epsilon,P_1} (p))$ as a {\em vertical cylindrical 
%neighborhood of the 
%crossing point $p$}.  If the value of $\epsilon$ is close to 0, we shall call it 
%a {\em small} vertical cylindrical neighborhood.  

\begin{figure}
\centerline{
        \hbox{
		\psfig{figure=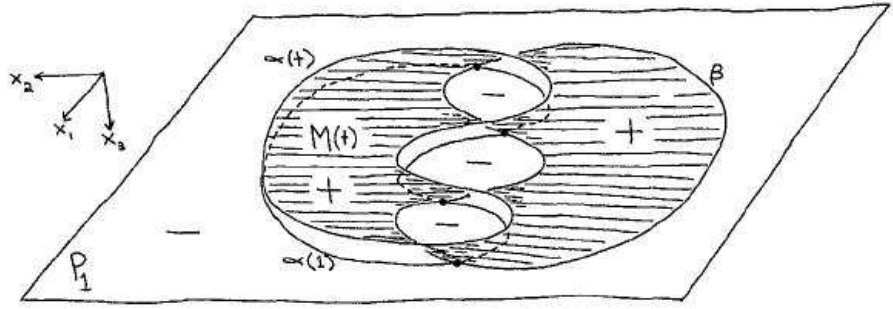,width=5in}
	}
}
	\caption{The Least-Area Surface $M(t)$ for $t$ close to $1$}
\end{figure}

\begin{theorem}
Suppose that $\cal W$ is the unique least-area set in $P_1$ with boundary equal 
to $\alpha(1) \cup \beta$.  Let $\{p_1,...,p_n\} = \alpha(1) \cap \beta$ be 
the crossing points.  Then there exists $\epsilon>0$ and $t_0 \in (0,1)$ so 
that for all $t \in [t_0,1)$, there is an orientable minimal 
surface $M(t) \subset 
\bfR^3$ with boundary $\alpha(t) \cup \beta$ having the following properties: 
\begin{list}%
{\arabic{num})}{\usecounter{num}\setlength{\rightmargin}{\leftmargin}}
\item $M(t)$ has least-area amongst all orientable 
surfaces with the same boundary.  
\item $M(t)$ is embedded.  
\item $M(t) \setminus C(\cup_{i=1}^n B_{\epsilon,P_1} (p_i))$ is 
a graph over the set ${\cal W} \subset P_1$.  
\item For each $i$, $M(t) \cap C(B_{\epsilon,P_1} (p_i))$ is homeomorphic to 
a disk with total absolute curvature less than $2 \pi$.  
\item $M(t)$ has genus $\frac{n-s}{2}$.  
\end{list}
\end{theorem}

\begin{remark}
The assumption that $M(t)$ is orientable 
is natural, since $\alpha(t) \cup \beta$ is extremal
and therefore any compact non-orientable minimal 
surface with this boundary cannot be 
embedded.  And any surface which is not embedded cannot be least-area, since one
can easily decrease area by adding topology and desingularizing at intersection 
points.  
\end{remark}

\begin{remark}
The assumption that $\cal W$ is the least-area set 
in $P_1$ with boundary equal to $\alpha(1) \cup \beta$ is really necessary.  
For $\alpha$ and $\beta$ for which this assumption does not 
hold, it follows from the results of \cite{HS} and 
\cite{D} that there exists an embedded minimal surface with the same boundary 
and with area 
strictly less than that of the $M(t)$ described above.  Again, 
embeddedness of this new surface with lesser area implies that it is also 
orientable.  
\end{remark}

\begin{remark}
With further arguments, one can show that in a small vertical neighborhood of 
each crossing point, $M(t)$ is "approximately helicoidal" in the following 
sense:  
For each crossing point $p_j \in P_1$, there exists a portion $S 
\subset \bfR^3$ of a helicoid 
bounded by a pair of infinite lines and with total absolute curvature less than 
$2 \pi$, and there exist homotheties $\phi_t$ centered at $p_j$ for 
each $t \in (0,1)$, such that
for any sequence $t_i \in (0,1)$ with $t_i \rightarrow 1$, 
$\{\phi_{t_i}(M(t_i))\}_{i=1}^\infty$ converges to $S$.  
(We say that a sequence of surfaces $\{S_i\}_{i=1}^\infty$ {\em converges} 
as $i \rightarrow \infty$ to 
a surface $S$ in $\bfR^3$ if, for any compact region $B$ of $\bfR^3$, 
there exists an integer $N_B$ such that for $i > N_B$, 
$S_i \cap B$ is a normal graph over $S$, and 
$\{ S_i \cap B \}_{i=1}^\infty$ converges to $S$ in the $C^1$-norm.)  
\end{remark}

Before giving the proof of Theorem 3.1, we give an useful preliminary 
lemma.  For an oriented surface ${\cal M} \subset \bfR^3$, 
let dist$_{\cal M}(A,B)$ be the distance in 
$\cal M$ between two sets $A \subset {\cal M}$ and $B \subset {\cal M}$.
Let $\partial {\cal M}$ be the boundary of $\cal M$.  
For each point $q \in {\cal M}$, let $K_q$ be the Gaussian curvature of 
$\cal M$ at $q$, and let $\vec{N}_q$ be the oriented unit normal vector 
of $\cal M$ at $q$.  Let $\vec{e}_3 = (0,0,1)$, and let $\langle \cdot , 
\cdot \rangle$ be the standard inner product on $\bfR^3$.  

\begin{lemma}
Let $\cal M$ be an oriented minimal 
surface lying in the closed region 
$\{t \leq x_3 \leq 1 \}$ between the two 
distinct planes $P_t$ and $P_1$, and consider a point $p \in {\cal M}$.  
Suppose there exists a positive constant $\cal A$ such that 
dist$_{\cal M}(p,\partial {\cal M}) \geq {1}/{\sqrt{\cal A}}$ and 
$|K_q| < {\cal A}$ for 
all $q \in {\cal M}$.  Then, for ${\cal T} := 
1 - 2 \sqrt{\cal A} (1-t)$, 
\[ | \langle \vec{N}_p , \vec{e}_3 \rangle |^2 > {\cal T} \; . \]
\end{lemma}

\begin{proof}
Note that ${\cal T} < 1$, and that 
the result is obvious if ${\cal T} < 0$, so we assume ${\cal T} \in 
[0,1)$.  
Suppose that $| \langle \vec{N}_p , \vec{e}_3 \rangle |^2 \leq 
{\cal T}$.  Then there exists a unit vector $\vec{T} \in T_p{\cal M}$ so 
that $\langle \vec{T}, \vec{e}_3 \rangle \geq \sqrt{1-{\cal T}}$, 
and there exists a unit speed geodesic $\gamma(s) \subset {\cal M}, 
s \in [0,{1}/{\sqrt{\cal A}}]$ so 
that $\gamma(0) = p$ and $\gamma^\prime (0) = 
\vec{T}$, where $\prime = \frac{\partial}{\partial 
s}$.  Let $k_g(s)$ be the 
geodesic curvature of $\gamma(s)$.  Since 
$|K_q| < {\cal A}$ and $\cal M$ is minimal, $|k_g(s)| < 
\sqrt{\cal A}$ for all $s \in [0,{1}/{\sqrt{\cal A}}]$.  Thus 
$|\gamma^{\prime 
\prime}(s)| < \sqrt{\cal A}$.  Writing $\gamma(s) = (\gamma_1(s),\gamma_2(s), 
\gamma_3(s))$ in terms of coordinates in $\bfR^3$, we have 
$|\gamma_3^{\prime \prime}(s)| < \sqrt{\cal A}$.  Then for 
$s \in [0,{1}/{\sqrt{\cal A}}]$, 
\[ |\gamma_3^\prime(s) - \gamma_3^\prime(0)| = 
\left| \int_0^s \gamma_3^{\prime \prime} \right| \leq 
\int_0^s |\gamma_3^{\prime \prime} | < \sqrt{\cal A} \cdot s \; \; , 
\]
and thus $\gamma_3^\prime(s) > \gamma_3^\prime(0) - 
\sqrt{\cal A} \cdot s$.  Therefore 
\[ \int_0^\frac{\gamma_3^\prime(0)}
{\sqrt{\cal A}} 
\gamma_3^\prime(s) ds > \int_0^\frac{\gamma_3^\prime(0)}
{\sqrt{\cal A}} (\gamma_3^\prime(0) - 
\sqrt{\cal A} \cdot s) ds = 
\gamma_3^\prime(0) \cdot 
\frac{\gamma_3^\prime(0)}
{\sqrt{\cal A}} - \frac{1}{2} \sqrt{\cal A} 
\left( \frac{\gamma_3^\prime(0)}
{\sqrt{\cal A}} \right)^2 
\; \; , \]
and so 
\[ \gamma_3 \left( \frac{\gamma_3^\prime(0)}
{\sqrt{\cal A}} \right) - \gamma_3(0) > 
\frac{(\gamma_3^\prime(0))^2}{2\sqrt{\cal A}} \geq 
\frac{1 - {\cal T}}{2\sqrt{\cal A}} = 1 - t \; \; , 
\]
since $\gamma_3^\prime(0) = \langle \vec{T}, \vec{e}_3 \rangle \geq 
\sqrt{1 - {\cal T}}$.  Hence the 
vertical change $\gamma_3({\gamma_3^\prime(0)}/{\sqrt{\cal A}}) 
- \gamma_3(0)$ of the geodesic 
$\gamma(s) \subset {\cal M}$ is greater than $1-t$, and 
since $1-t$ is the distance between the planes $P_t$ and $P_1$, 
$\cal M$ cannot lie between $P_t$ and $P_1$.  
This contradiction proves the lemma.  
\end{proof}

We now give the proof of Theorem 3.1:

\begin{proof}
It was shown in \cite{HS} that there exists an orientable minimal 
surface $M(t)$ with boundary $\alpha(t) \cup \beta$ that 
has least-area amongst all orientable surfaces with the same boundary, and 
that $M(t)$ is embedded, with finite genus, 
and with bounded curvature.  $M(t)$ cannot have any 
interior branch points (\cite{O2}), and since its boundary is extremal, 
it cannot have any boundary branch points (\cite{N2}, section 366).  

We divide the proof into steps.  

We fix $\epsilon$ to be a small positive number such that $\frac{1}{\epsilon}$ is 
much greater than the maximum planar curvature of $\alpha \cup \beta$.  (In steps 
3 and 6 we will add further constraints on $\epsilon$.)  Let 
$N_{\epsilon,P_1}(\alpha(1) \cup \beta) 
:= \{ p \in P_1 \; 
| \; \mbox{dist}(p,\alpha(1) \cup \beta) < \epsilon \}$ be 
the $\epsilon$ neighborhood of $\alpha(1) \cup \beta$ in $P_1$.  

{\bf Step 1:} {\em $M(t) \setminus C(N_{\epsilon,P_1}(\alpha(1) \cup 
\beta))$ is a collection of 
graphs over $P_1$ for $t$ sufficiently close to $1$.  In fact, for 
any $y \in (0,1)$, there exists a $t_y \in (0,1)$ such that 
for all $t \in [t_y,1)$, 
$| \langle \vec{N}_p , \vec{e}_3 \rangle | > y$ for all $p \in 
M(t) \setminus C(N_{\epsilon,P_1}(\alpha(1) \cup \beta))$.  }

%If this were not so, there would exist a point $q \in 
%C(P_1 \setminus N_{\epsilon,P_1}(\alpha(1) \cup \beta)) \cap M(t)$ such 
%that $\vec{e}_3 \in T_q(M(t))$.  Let $\gamma(s)$ be a unit-speed geodesic in 
%$M(t)$ such that $\gamma(0) = q$ and $\gamma^\prime(0) = \vec{e}_3$ ($ ^\prime 
%= \frac{d}{ds}$).  
%Since $M(t)$ is contained between the planes $P_t$ and $P_1$, the vertical change 
%in $\gamma$ can be at most $1-t$.  We will show that the vertical change in 
%$\gamma$ must be greater than $1-t$ if $t$ is sufficiently close to 1, and 
%this contradiction implies step 1.  

By Corollary 4 of 
\cite{Sc2} there exists a universal constant 
$c \geq 1$ such that $|K_q| < \frac{c}{r^2}$, where $K_q$ is the Gaussian curvature 
at some point $q$ of a stable minimal surface and $r$ is the distance 
from $q$ to the boundary of the surface.  Thus for all $q \in 
M(t) \setminus C(N_{\epsilon/2,P_1}(\alpha(1) \cup \beta))$, 
$|K_q| < \frac{4c}{\epsilon^2}$.  We now apply Lemma 3.1 with 
${\cal M} = M(t) \setminus C(N_{\epsilon/2,P_1}(\alpha(1) \cup \beta))$ and 
${\cal A} = \frac{4c}{\epsilon^2}$.  Let $p$ be any point in 
$M(t) \setminus C(N_{\epsilon,P_1}(\alpha(1) \cup \beta))$.  Note 
that dist$_{\cal M} 
(p,\partial {\cal M}) \geq \epsilon /2 = \sqrt{c} / \sqrt{\cal A} \geq 
1/ \sqrt{\cal A}$.  Hence by Lemma 3.1, 
\[ | \langle \vec{N}_p , \vec{e}_3 \rangle |^2 > 1 - 2 
\frac{2 \sqrt{c}}{\epsilon} (1-t) \; \; . \] 
So for any $t \in [1 - \frac{\epsilon}{4 \sqrt{c}}, 1)$, we have 
$| \langle \vec{N}_p , \vec{e}_3 \rangle |^2 > 0$.  Thus 
for $t \in [1 - \frac{\epsilon}{4 \sqrt{c}}, 1)$ the normal vector is 
never horizontal on $M(t) \setminus C(N_{\epsilon,P_1}(\alpha(1) \cup 
\beta))$, implying that 
$M(t) \setminus C(N_{\epsilon,P_1}(\alpha(1) \cup \beta))$ is 
a collection of graphs.  Furthermore, for 
any $y \in (0,1)$, let $t_y = 1 - (1-y^2) (\epsilon / (4\sqrt{c}) )$, and then 
$| \langle \vec{N}_p , \vec{e}_3 \rangle | > y$ for all $p \in 
M(t) \setminus C(N_{\epsilon,P_1}(\alpha(1) \cup \beta))$ if 
$t \in [t_y,1)$.  Step 1 is shown.  

%This implies that the geodesic curvature 
%$k_g (\gamma(s))$ of $\gamma(s)$ 
%is bounded above by $\frac{2\sqrt{c}}{\epsilon}$ for 
%$s \in [0,\frac{\epsilon}{2})$.  Thus the third coordinate 
%$\gamma_3^{\prime \prime}$ of $\gamma^{\prime \prime} (s)$ satisfies 
%$| \gamma_3^{\prime \prime} | < 
%\frac{2\sqrt{c}}{\epsilon}$ for $s \in [0,\frac{\epsilon}{2})$.  Since 
%$\gamma_3^\prime (0) = 1$, we then have $\gamma_3^\prime (s) > 1 - 
%\frac{2 s \sqrt{c}}{\epsilon}$ for $s \in [0,\frac{\epsilon}{2})$.  It follows that 
%the vertical change 
%$\gamma_3(\frac{\epsilon}{2 \sqrt{c}}) - \gamma_3(0) > \frac{\epsilon}
%{4 \sqrt{c}} > 1-t$, whenever $ t > 1 - \frac{\epsilon}{4 \sqrt{c}}$.  

{\bf Step 2:} {\em $M(t) \setminus C(N_{\epsilon,P_1}(\alpha(1) \cup 
\beta))$ is a single 
graph over $P_1$ for $t$ sufficiently close to $1$.}

Suppose that for some point $q \in P_1 
\setminus N_{\epsilon,P_1}(\alpha(1) \cup \beta)$, 
$C(B_{\frac{\epsilon}{2},P_1} (q) 
\setminus N_{\epsilon,P_1}(\alpha(1) \cup \beta))$ contains two graphs.  
These two graphs are bounded by a pair of curves 
$C_1$ and $C_2$ in the boundary 
$\partial ( C(B_{\frac{\epsilon}{2},P_1} 
(q) \setminus N_{\epsilon,P_1}(\alpha(1) \cup \beta)) )$; 
and $C_1 \cup C_2$ must lie between 
$P_t$ and $P_1$, since $M(t)$ itself lies between $P_t$ and $P_1$.  Note that 
this pair of graphs has area at least $2 \cdot 
\mbox{Area}(B_{\frac{\epsilon}{2},P_1} 
(q) \setminus N_{\epsilon,P_1}(\alpha(1) \cup \beta))$.  Since 
$M(t)$ is least-area, this pair of graphs must be least-area with respect to 
its boundary $C_1 \cup C_2$.  However, $C_1 \cup C_2$ also bounds an annulus that 
has area less than $(1-t) \mbox{Length}(\partial (B_{\frac{\epsilon}{2},P_1} 
(q) \setminus N_{\epsilon,P_1}(\alpha(1) \cup \beta)))$, and 
$(1-t) \mbox{Length}(\partial (B_{\frac{\epsilon}{2},P_1} 
(q) \setminus N_{\epsilon,P_1}(\alpha(1) \cup \beta))) < 
2 \mbox{Area}(B_{\frac{\epsilon}{2},P_1} 
(q) \setminus N_{\epsilon,P_1}(\alpha(1) \cup \beta))$, 
if $t$ is sufficiently close to 1.  
This contradiction implies step 2.  

{\bf Step 3:} {\em $C({\cal W} \setminus 
N_{\epsilon,P_1}(\alpha(1) \cup \beta) ) \cap M(t) = 
M(t) \setminus C(N_{\epsilon,P_1}(\alpha(1) \cup \beta))$.}

Let 
$\tilde{\cal W}$ be the union of open 
components of $P_1 \setminus \{ \alpha(1) \cup \beta \}$ in $P_1$ whose 
intersections with 
the projection of $M(t) \setminus C(N_{\epsilon,P_1}(\alpha(1) \cup 
\beta))$ are nonempty.  Clearly 
the boundary of $\tilde{\cal W}$ in $P_1$ equals $\alpha(1) \cup \beta$, thus 
Area($\tilde{\cal W}$) $>$ Area ($\cal W$) if $\tilde{\cal W} \neq 
{\cal W}$, by assumption.  Suppose $\tilde{\cal W} \neq {\cal W}$.  

Since $M(t) \setminus C(N_{\epsilon,P_1}(\alpha(1) \cup \beta))$ is 
a graph over 
$\tilde{\cal W} \setminus N_{\epsilon,P_1}(\alpha(1) \cup \beta)$, we have 
Area($M(t)$)$\geq$Area($\tilde{\cal W} \setminus N_{\epsilon,P_1}(\alpha(1) 
\cup \beta)$)$>$Area($\tilde{\cal W}$)$-2 \epsilon (\mbox{Length}(\alpha) + 
\mbox{Length}(\beta))$.  (The final strict inequality follows from the fact 
that $\frac{1}{\epsilon}$ is 
much greater than the maximum planar curvature of $\alpha \cup \beta$.)  
Choosing $\epsilon$ smaller if necessary, we may 
assume $2 \epsilon (\mbox{Length}(\alpha) + 
\mbox{Length}(\beta)) < (1/2)$(Area($\tilde{\cal W}$)$ - $Area(${\cal W}$)).  
Therefore 
\[ 
\mbox{Area}(M(t))>\mbox{Area}(\tilde{\cal W})- 
\frac{1}{2} (\mbox{Area}(\tilde{\cal W}) - \mbox{Area}({\cal W})) \; \; .  
\]
For any $\delta > 0$, there exists a $t_\delta \in (0,1)$ so that 
for all $t \in [t_\delta, 1)$, there is an orientable surface 
$\breve{M}(t)$ bounded by $\alpha(t) \cup \beta$ satisfying 
Area($\breve{M}(t)$) $<$ Area($\cal W$)$+\delta$: just consider a surface that is a 
graph away from crossing points 
over each open component of $P_1 \setminus (\alpha(1) \cup \beta)$ assigned 
a plus sign, and connect these graphs by 
small disks twisted by approximately $\pi$ radians at each crossing point.  
Choosing $\delta = (1/2)$(Area($\tilde{\cal W}$)$ - $Area(${\cal W}$)), and 
choosing $t \geq t_\delta$, we have 
\[ 
\mbox{Area}(M(t))>\frac{1}{2} (\mbox{Area}(\tilde{\cal W}) + 
\mbox{Area}({\cal W}))>\mbox{Area}(\breve{M}(t)) \; \ .  
\] 
But $M(t)$ is least-area, so this contradiction implies 
$\tilde{\cal W} = {\cal W}$.  This shows step 3.  

To prepare for step 4, consider 
$N_{\epsilon,P_1}(\alpha(1) \cap \beta) := \{ p \in P_1 \; 
| \; \mbox{dist}(p,\alpha(1) \cap \beta) < \epsilon \}$, 
the $\epsilon$ neighborhood in $P_1$ of the crossing points.  

{\bf Step 4:} {\em Along the curves $(\alpha(t) \cap 
\beta) \setminus C(N_{\epsilon,P_1}(\alpha(1) \cap \beta))$ the 
normal vector to $M(t)$ must 
become arbitrarily close to vertical, uniformly as $t \rightarrow 1$.}  

To show this, for $t$ sufficiently close to $1$ we 
consider a compact connected piece of a catenoid 
${\cal C}(r_1,r_2,t)$ with the following properties: 
\begin{itemize}
\item ${\cal C} = {\cal C}(r_1,r_2,t)$ has vertical axis and is 
bounded by circles 
$P_t \cap {\cal C}$ and $P_1 \cap {\cal C}$, of radii $r_2$ and 
$r_1$ respectively, in the planes $P_t$ and $P_1$.  
\item $r_2$ is much smaller than $r_1$, and $r_1$ is much smaller than $\epsilon$.  
\item The normal vector to $\cal C$ is close to vertical everywhere on $\cal C$.  
\end{itemize}
Choosing $r_2$ and $r_1$ small enough 
and choosing $t$ close enough to 1, $\cal C$ can be placed in the 
vertical cylinder over any open component of $P_1 \setminus 
(\alpha(1) \cup \beta)$ assigned a minus sign, 
so that ${\cal C} \cap M(t) = \emptyset $, by step 3.  
We can then translate $\cal C$ horizontally until it makes first 
contact with $M(t)$ at any given point $p \in 
\alpha(t) \setminus C(N_{\epsilon,P_1}(\alpha(1) \cap \beta))$.  
(The maximum principle implies that first contact cannot be at an 
interior point of $M(t)$.  See Figure 3.)  By the maximum principle, 
the normal vector of $M(t)$ at $p$ must be even 
more vertical than it is at the same point on $\cal C$.  Hence the normal vector of 
$M(t)$ at $p$ is close to vertical.  
Reflecting $\cal C$ through the plane $P_{\frac{1+t}{2}}$, the same argument shows
that the normal vector of $M(t)$ any point of 
$\beta \setminus C(N_{\epsilon,P_1}(\alpha(1) \cap \beta))$ must 
also be close to vertical.  

Additionally, we have just shown 
that the projection of $M(t) \setminus C(N_{\epsilon,P_1}(\alpha(1) \cap 
\beta))$ to $P_1$ lies entirely inside 
$\cal W$, for otherwise $\cal C$ (or its reflection through 
$P_{\frac{1+t}{2}}$) could be translated in such a way as to 
make first contact at an interior point.  

At that first boundary 
point $p \in \partial M(t)$ of contact, the catenoid 
$\cal C$ (or its reflection through 
$P_{\frac{1+t}{2}}$) makes a small angle with the 
horizontal plane containing the boundary curve.  Among all catenoid pieces of 
the type ${\cal C}(r_1,r_2,t)$ contacting 
$M(t)$ only at $p$ there is a greatest lower bound 
$\theta(p,t)$ for this angle, and 
for all boundary points $p \in (\alpha(t) \cup \beta) 
\setminus C(N_{\epsilon,P_1}(\alpha(1) \cap \beta))$, the lower bound 
$\theta(p,t)$ approaches zero as $t$ approaches 1 (see Figure~3).  Let 
\[ \theta_0(t) = \max\{\theta(p,t) \; : \; p \in 
(\alpha(t) \cup \beta) 
\setminus C(N_{\epsilon,P_1}(\alpha(1) \cap \beta)) \} \; . \] 
Since $(\alpha(t) \cup \beta) 
\setminus C(N_{\epsilon,P_1}(\alpha(1) \cap \beta))$ is compact, 
$ \lim_{t \rightarrow 1} \theta_0(t) = 0$.   

\begin{figure}
\centerline{
        \hbox{
		\psfig{figure=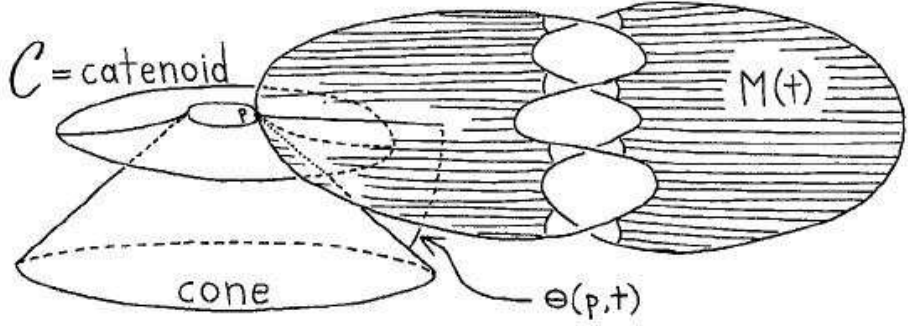,width=5in}
	}
}
	\caption{A Pictoral View of $\theta(p,t)$}
\end{figure}

{\bf Step 5:} {\em $M(t) 
\setminus C(N_{2\epsilon,P_1}(\alpha(1) \cap \beta))$ is a graph 
over $\cal W$, for $t$ sufficiently close to $1$.  
In fact, for any $y \in (0,1)$, there exists a $t_y \in (0,1)$ such that 
if $t \in [t_y,1)$, then $|\langle \vec{N}_p, 
\vec{e}_3 \rangle| > y \; \forall p \in M(t) 
\setminus C(N_{2\epsilon,P_1}(\alpha(1) \cap \beta))$. }

Since we have already proved this for $M(t) 
\setminus C(N_{\epsilon,P_1}(\alpha(1) \cup \beta))$ in steps 
1 and 2, we may restrict our considerations here to any points $p \in 
M(t) \cap C(N_{\epsilon,P_1}(\alpha(1) \cup \beta) 
\setminus N_{2\epsilon,P_1}(\alpha(1) \cap \beta))$.  
Let $\hat{p}$ be 
the vertical projection of $p$ into $P_1$, and let $r$ be the 
distance of $\hat{p}$ to $\alpha(1) \cup \beta$.  (Note that $r \leq 
\epsilon$.)  Assume for the moment that the closest point in $\alpha(t) 
\cup \beta$ to $p$ is a point in $\beta$.  

Consider the bounded cylinder $Cyl =
C(B_{\frac{r}{2},P_1}(\hat{p})) \cap \{ 1 - 2r \tan{\theta_0}(t) \leq x_3 
\leq 1\}$.  $Cyl$ is a finite 
solid cylinder 
of radius $\frac{r}{2}$ and height $2 r \tan{\theta_0}(t)$ with two planar 
horizontal circular 
boundary disks.  Let $\partial Cyl$ be its boundary.  
It follows from the construction of the catenoid 
barriers in the previous step (see Figure 3) that $M(t) 
\cap \partial Cyl$ is contained entirely in the cylindrical part 
of $\partial Cyl$, and is disjoint from the upper and 
lower horizontal boundary disks.  
To avoid unnecessary complications near the 
crossing points, we have increased $\epsilon$ to $2\epsilon$.  

Let ${\cal M} = Cyl \cap M(t)$.  Then $\cal M$ lies between the planes 
$P_{1-2r \tan{\theta_0}(t)}$ and $P_1$.  Since for any point $q \in {\cal M}$, 
dist$_{M(t)} (q, \alpha(t) \cup \beta) \geq r/2$, 
there exists a universal constant 
$c \geq 1$ such that $|K_q| < \frac{4c}{r^2}$ \cite{Sc2}, like 
in step 1.  Let ${\cal A} = 
4c/(r^2)$, and thus dist$_{\cal M} (p, \partial {\cal M}) \geq 
r/2 = \sqrt{c}/\sqrt{\cal A} \geq 1/\sqrt{\cal A}$.  Hence we may apply 
Lemma 3.1 to conclude that 
\[
|\langle \vec{N}_p, 
\vec{e}_3 \rangle|^2 > 1 - 2 \frac{2\sqrt{c}}{r} 
2r \tan{\theta_0}(t) \; \; . 
\]
For any $y \in (0,1)$, choose $t_y \in (0,1)$ so that 
for all $t \in [t_y,1)$, 
\[ \tan{\theta_0}(t) \leq \frac{1-y^2}{8\sqrt{c}} \; \; . \] 
Then, for $t \in [t_y,1)$, $|\langle \vec{N}_p, 
\vec{e}_3 \rangle| > y \; \forall p \in M(t) 
\setminus C(N_{2\epsilon,P_1}(\alpha(1) \cap \beta))$.  

When the closest point in $\alpha(t) \cup \beta$ to $p$ is 
a point in $\alpha(t)$, we can reflect $M(t)$ across the plane 
$P_{(1+t)/2}$ and we have the same situation, but the roles of 
$\alpha(t)$ and $\beta$ are reversed.  So the above argument applies to 
this case as well.  Step 5 is shown.  

%Choose a unit speed geodesic $\gamma(s) \subset M(t)$ such that $\gamma(0) 
%= p$ and $\gamma^\prime(0) = \vec{e}_3$.  
%Note that since $\tan{\theta_0}(t) < \frac{1}{8}$, $\gamma(s)$ lies in 
%$B_{\frac{r}{2},P_1}(\hat{p}) \times \bfR$ for $s \in [0,4r\tan{\theta_0}(t))$.  
%Thus the geodesic curvature $|\gamma^{\prime\prime}(s)|$ is 
%bounded above by $\frac{2\sqrt{c}}{r}$ for 
%$s \in [0,4r\tan{\theta_0}(t))$, by Schoen's estimate.  
%It follows, just as in the argument of step 1, that the geodesic 
%$\gamma(s)$ must intersect the upper circular flat disk of $\partial 
%Cyl$.  Thus $M(t)$ intersects the upper flat disk of $\partial 
%Cyl$ as well.  
%This contradicts the way $Cyl$ was constructed, hence such a point 
%$p$ cannot exist.  
%Therefore the tangent planes $T_p(M(t))$ are never vertical 
%on $M(t) \cap C(P_1 \setminus N_{2\epsilon,P_1}(\alpha(1) \cap 
%\beta))$, and this shows step 5.  

%Choose $t$ close enough 
%to 1 so that $\tan{\theta_0}(t) < \min\{\frac{1}{16\sqrt{c}},\frac{1}{8}\}$.  
%Suppose $p$ is a point in $M(t) \cap 
%C(P_1 \setminus N_{2\epsilon,P_1}(\alpha(1) \cap \beta))$ such that 
%$T_p (M(t))$ is vertical.    

{\bf Step 6:} {\em For each crossing point $p_i$, 
$C(B_{2\epsilon,P_1} (p_i) ) \cap  
M(t)$ is a disk with total curvature less than $2\pi$, for $t$ sufficiently 
close to $1$.}  

We could have originally chosen $\epsilon$ small enough 
so that the two curve segments $\alpha(t) \cap 
C(B_{2\epsilon,P_1} (p_i))$ and 
$\beta \cap C(B_{2\epsilon,P_1} (p_i))$ have 
total curvature as close to zero as we wish.  We may then 
choose $t$ sufficiently close 
to $1$ so that $\partial (M(t) \cap C(B_{2\epsilon,P_1} (p_i))) \setminus 
\{ \alpha(t) \cup \beta \}$ approximates arbitrarily closely a pair of 
circular arcs, each 
of less than $\pi$ radians, and so that the exterior angles at the 
four singular points of $\partial (M(t) \cap C(B_{2\epsilon,P_1} (p_i)))$ are 
each arbitrarily close to 
$\frac{\pi}{2}$.  (This follows from the fact, as shown in steps 1 and 5, 
that for any given $y \in (0,1)$ we have $|\langle \vec{N}_q, 
\vec{e_3} \rangle| > y$ on $M(t) \setminus 
C(N_{2\epsilon,P_1}(\alpha(1) \cap \beta))$ for 
$t$ sufficiently close to 1.)  So for $\epsilon$ small enough 
and $t$ sufficiently close to 
$1$, $\partial (M(t) \cap 
C(B_{2\epsilon,P_1} (p_i)))$ has total curvature less than $4\pi$.  

If $M(t) \cap C(B_{2\epsilon,P_1} (p_i))$ is not a disk, then 
there exists a smooth Jordan curve $\sigma \subset$ Int($M(t) \cap 
C(B_{2\epsilon,P_1} (p_i))$) with the 
following properties:
\begin{itemize}
\item $\sigma$ is a smooth approximation to $\partial (M(t) \cap 
C(B_{2\epsilon,P_1} (p_i)))$ that lies on the 
boundary of a convex region in $\bfR^3$ and 
has total curvature less than $4\pi$.  (This property is possible because 
$\partial (M(t) \cap C(B_{2\epsilon,P_1} (p_i)))$ is extremal.)  
\item $\sigma$ separates $M(t) \cap 
C(B_{2\epsilon,P_1} (p_i))$ into two components, one component 
${\cal A}_1$ is bounded by $\sigma$, and the other component ${\cal A}_2$ is 
an annulus bounded by $\sigma \cup \partial (M(t) \cap 
C(B_{2\epsilon,P_1} (p_i)))$.  
\item The component ${\cal A}_1$ is not a disk.  
\end{itemize} 
However, section 3 of \cite{MY} then implies that 
${\cal A}_1$ must be a disk.  This contradiction implies that $M(t) \cap 
C(B_{2\epsilon,P_1} (p_i))$ is 
actually a disk.  The Gauss-Bonnet theorem then implies that the total 
absolute curvature satisfies 
\[ \int_{M(t) \cap C(B_{2\epsilon,P_1} (p_i))} |K| dA < 2\pi \; . \]  
This shows step 6.  Dividing our original choice for $\epsilon$ by $2$, 
step 6 implies the fourth item of the theorem.

Thus $M(t)$ with $t$ close to $1$ is a graph away from crossing points 
over components of $P_1 \setminus 
\{ \alpha(1) \cup \beta \}$ assigned a plus sign.  And 
$M(t)$ is a disk in small vertical neighborhoods of each 
crossing point.  One can then easily check 
that the Euler characteristic of 
$M(t)$ is $s-n$, and therefore the genus of $M(t)$ is $\frac{n-s}{2}$.
\end{proof}

%\begin{remark}
%The convergence of the homothetically expanded surfaces near the 
%crossing points to a portion of a helicoid (step 8) is in fact 
%convergence in the $C^\infty$-topology.  
%One way to show this is by using elliptic regularity.  Another way 
%is to parametrize the surfaces by their Gauss maps.  Then the coordinate 
%functions are harmonic and are converging to limit harmonic 
%functions.  It follows that all derivatives of all orders must also 
%converge.  
%\end{remark}

Just as the Meeks and White result was extended to the case of a wedge by 
Theorem~2.1 in the previous section, 
likewise Theorem~3.1 can be extended 
to the case of a wedge.  This is done in Theorem~3.2, and the proof 
is essentially the same as 
for Theorem~3.1.  Let $\beta$ be a $C^2$ Jordan curve in the interior 
of the half-plane $\omega(0)$, and let $\alpha$ be a $C^2$ Jordan curve in 
the interior of the half-plane $\omega(\frac{\pi}{2})$.  Let 
$\alpha(t) = R_{\frac{-\pi t}{2}}(\alpha)$ be the rotation 
about the $x_1$-axis of $\alpha$ into the half-plane 
$\omega(\frac{\pi}{2} (1-t) )$.  Suppose $\alpha(1)$ and $\beta$ 
intersect transversely at a finite number of points.  Call these $n$ points 
the {\em crossing points} of $\alpha$ and $\beta$.  
Define $s$ and $\cal W$ just as before.  

\begin{theorem}
Suppose 
that $\cal W$ is the unique least-area set in $\omega(0)$ with boundary equal 
to $\alpha(1) \cup \beta$.  Let $\{p_1,...,p_n\} = \alpha(1) \cap \beta$ be 
the crossing points.  Then there exists $\epsilon>0$ and $t_0 \in (0,1)$ so 
that for all $t \in [t_0,1)$, there is an orientable minimal 
surface $M(t) \subset 
\bfR^3$ with boundary $\alpha(t) \cup \beta$ having the following properties:  
\begin{list}%
{\arabic{num})}{\usecounter{num}\setlength{\rightmargin}{\leftmargin}}
\item $M(t)$ has least area amongst all orientable surfaces with the same boundary.  
\item $M(t)$ is embedded.  
\item $M(t) \setminus C(\cup_{i=1}^n B_{\epsilon,\omega(0)} (p_i))$ is 
a graph over the set ${\cal W} \subset \omega(0)$.  
\item For each $i$, $M(t) \cap C(B_{\epsilon,\omega(0)} (p_i))$ is 
homeomorphic to a disk with total absolute 
curvature less than $2 \pi$.  
\item $M(t)$ has genus $\frac{n-s}{2}$.  
\end{list}
\end{theorem}

\section{Non-Existence Results} 

	The results we prove in this section are about existence of 
compact minimal and constant mean curvature surfaces with a given pair 
of boundary curves $C_1, C_2$ in parallel horizontal 
planes.  When $C_1$ and $C_2$ are convex and the vertical projection 
of $C_1$ into the plane containing $C_2$ does not intersect $C_2$, 
%Schoen has proven that $M$ does not exist.might reference personal 
%communication with Rob here.
it is well known that such a minimal surface does not exist.  
For completeness we give the proof here.  

\begin{lemma}
	If $C_1$ and $C_2$ are convex 
curves in parallel planes and if the perpendicular projection 
of $C_1$ into the plane containing $C_2$ is disjoint from $C_2$, then there 
does not exist a compact connected minimal surface with boundary 
$C_1 \cup C_2$.
\end{lemma}

\begin{proof}
	We may assume $C_1$ and $C_2$ are contained in planes parallel to the 
$x_1x_2$-plane, and by hypothesis we may assume 
$C_1$ and $C_2$ lie on opposite sides of and are disjoint 
from the plane $\{x_1 = 0\}$.  Suppose there exists a compact connected 
minimal surface $M$ with boundary $C_1 \cup C_2$, then let 
$\mbox{Ref}(M)$ be the 
reflection of $M$ through the plane $\{x_1 = 0\}$.  Translate $\mbox{Ref}(M)$ 
in the 
direction of the $x_2$-axis until it is disjoint from $M$, then translate it 
back until the first moment when $M$ and $\mbox{Ref}(M)$ 
intersect.  This point of 
intersection must be in the interiors of both $M$ and $\mbox{Ref}(M)$, 
and $M \neq \mbox{Ref}(M)$, contradicting the maximum principle.  
\end{proof}

We use the same kind of technique to prove the next lemma.  A {\em slab} in 
$\bfR^3$ is a region lying between two parallel planes.  A slab is 
called {\em vertical} if its two boundary 
planes are perpendicular to the $x_1x_2$-plane.  

\begin{lemma}
	Suppose $C_1$ and $C_2$ are two planar curves in horizontal planes, 
and these two planes are of distance $h_1$ apart.  If 
$C_1 \cup C_2$ is contained in a vertical 
slab of width $h_2 < h_1$, then there does not exist a compact 
connected minimal surface with boundary $C_1 \cup C_2$.
\end{lemma}

\begin{proof}
We may assume that the boundary planes of the vertical slab are parallel to the 
$x_1$-axis, 
and we may assume the $x_1$-axis is equidistant from the two planes 
containing $C_1$ and $C_2$ and also from the boundary planes of the vertical slab.  
Suppose there exists a compact connected minimal surface $M$ with boundary 
$C_1 \cup C_2$.  Let $R_{\frac{\pi}{2}}(M)$ be a 
rotation of $M$ by $\frac{\pi}{2}$ radians about 
the $x_1$-axis.  Translate $R_{\frac{\pi}{2}}(M)$ 
in the direction of the $x_1$-axis until it 
is disjoint from $M$, then translate it back until the first moment when 
$M$ and $R_{\frac{\pi}{2}}(M)$ intersect.  Since $h_2 < h_1$, this point of 
intersection must be in the interiors of both $M$ and $R_{\frac{\pi}{2}}(M)$, 
and $M \neq R_{\frac{\pi}{2}}(M)$, contradicting the maximum principle.  
\end{proof}

\begin{corollary}
	If $C_1$ and $C_2$ are planar curves 
with diameters $d_1$ and $d_2$ in parallel planes, and these two 
planes are of distance $h$ apart, and these two curves bound a compact 
connected minimal surface $M$, then $h \leq $ \mbox{max}$\{d_1, d_2\}$.  
\end{corollary}

\begin{proof}
	We may assume $C_1$ and $C_2$ lie in horizontal planes of distance 
$h$ apart. Let $P_1$ be a vertical plane that is tangent to both $C_1$ and 
$C_2$, so that $C_1 \cup C_2$ lies entirely to one side of $P_1$.  There 
exists another plane $P_2$ parallel to $P_1$ such that $C_1 \cup C_2$ lies 
entirely within the vertical slab bounded by $P_1$ and $P_2$, and the 
distance between $P_1$ and $P_2$ is at most $\max\{d_1,d_2\}$.  
The result follows from Lemma~4.2.  
\end{proof}

	As stated in the introduction, the 
above corollary is a generalization of a result by 
Nitsche \cite{N1}, where it is assumed that $M$ is an annulus and 
gives the weaker conclusion $h \leq \frac{3}{2}\max\{d_1, d_2\}$.  Even our 
result here perhaps does not give the strongest possible result for 
an upper bound on $h$.  For example, if $M$ is an annular 
subregion of a catenoid with a 
vertical axis in $\bfR^3$, and the boundary $\partial M$ is 
two circles (with radii $r_1$ and $r_2$, respectively) in horizontal planes, 
then the two horizontal planes containing $\partial M$ are of distance less 
than $\frac{2}{3}(r_1 + r_2)$ apart.  
One can easily see this by elementary 
considerations of the generating curve $y = \cosh(x)$ of the catenoid.  
Since $\frac{2}{3}(r_1 + r_2) \leq \frac{2}{3} \max\{d_1=2r_1,d_2=2r_2\}$, 
we know that Corollary~4.1 is not the strongest possible result in the 
case of this catenoid $M$.  

	Lemma 4.2 cannot be extended to surfaces of constant non-zero 
mean curvature, as the round cylinder shows, but something can be said about 
the possible values of the mean curvature in the case of nonminimal 
constant mean curvature surfaces:  

\begin{proposition}
	Let $C = \{x_1^2 + x_2^2 \leq r^2, 0 \leq x_3 \leq d\} \subset 
\bfR^3$, with $r < \frac{d}{2}$.  If $\Sigma \subset C$ 
is an embedded constant mean curvature 
surface with mean curvature $H$ (with respect to the inward pointing 
normal) and boundary in the planes $P_0 = 
\{x_3 = 0\}$ and $P_d = \{x_3 = d\}$, then 
\[ H > \frac{d^2 - 12r^2}{2r(d^2 - 4r^2)} - 
\frac{32r^2d^2}{(d^2 - 4r^2)^3} \; . \]  
Thus when $d$ is large relative to $r$, the mean curvature $H$ of $\Sigma$ is 
bounded away from zero.  Furthermore, the limiting value for this lower bound 
as $d \rightarrow \infty$ is equal to the mean curvature of a 
cylinder of radius $r$.  
\end{proposition}

\begin{proof}
	Consider the embedded 
annular surfaces in $C$ which are surfaces of rotation about the 
$x_3$-axis with boundary curves 
$\{x_1^2 + x_2^2 = r^2, x_3 = 0\}$ and
$\{x_1^2 + x_2^2 = r^2, x_3 = d\}$ and with generating curves that are arcs 
of a circle.  Among this 1-parameter family of surfaces the mean 
curvature (with respect to the inward pointing normal) is always larger than 
$\frac{d^2 - 12r^2}{2r(d^2 - 4r^2)} - 
\frac{32r^2d^2}{(d^2 - 4r^2)^3}$.  Shrinking this family of surfaces 
from a cylinder of radius $r$ to the surface which makes first contact with 
$\Sigma$, we may apply the maximum principle \cite{Sc1} to conclude the result.  
\end{proof}

\bibliographystyle{plain}

\vspace{0.5in}

\begin{flushright}
{\em Wayne Rossman}\\
{\em Graduate School of Mathematics}\\
{\em Kyushu University, Fukuoka 812}\\
{\em Japan}\\
\end{flushright}

\end{document}